\documentclass{article}

\usepackage{setspace}
\usepackage{amssymb,amsmath,amsthm,tabu,longtable}
\usepackage{subfigure}
\usepackage{calc}
\usepackage{verbatim}
\usepackage{enumerate}
\usepackage{epsfig}
\usepackage{graphicx}
\usepackage{graphics}
\usepackage {tikz}
\usetikzlibrary{automata,arrows,calc,positioning}
\usetikzlibrary{decorations.pathreplacing}
\usepackage{xcolor}\usepackage[hmargin=3cm,vmargin=3.5cm]{geometry}
\usepackage{chngcntr}
\counterwithin{figure}{section}

\newtheorem{definition}{Definition}[section]
\newtheorem{lemma}[definition]{Lemma}

\newtheorem{theorem}[definition]{Theorem}

                                {\null\hfill$\Box$\par\medskip\medskip\medskip\medskip}

\makeatletter
\renewcommand{\pod}[1]{\mathchoice
  {\allowbreak \if@display \mkern 18mu\else \mkern 8mu\fi (#1)}
  {\allowbreak \if@display \mkern 18mu\else \mkern 8mu\fi (#1)}
  {\mkern4mu(#1)}
  {\mkern4mu(#1)}
}

\title{On the Real Roots of Domination Polynomials}

\author{Iain Beaton\thanks{Corresponding author}\\
\small Department of Mathematics \& Statistics\\[-0.8ex]
\small Dalhousie University\\[-0.8ex] 
\small Halifax, CA\\
\small\tt ibeaton@dal.ca\\
\and
Jason I. Brown\thanks{Supported by NSERC grant RGPIN-2018-05227}\\
\small Department of Mathematics \& Statistics\\[-0.8ex]
\small Dalhousie University\\[-0.8ex] 
\small Halifax, CA\\
\small\tt Jason.Brown@dal.ca\\
}

\begin{document}

\tikzset{bignode/.style={minimum size=3em,}}
\maketitle

\begin{abstract}
A dominating set $S$ of a graph $G$ of order $n$ is a subset of the vertices of $G$ such that every vertex is either in $S$ or adjacent to a vertex of $S$. The domination polynomial is defined by $D(G,x) = \sum d_k x^k$ where $d_k$ is the number of dominating sets in $G$ with cardinality $k$. In this paper we show that the closure of the real roots of domination polynomials is $(-\infty,0]$.
\end{abstract}

\setstretch{1.4}

\section{Introduction}\label{sec:intro}

For many graph polynomials, the location and nature of the roots have been (and continues to be) active areas of study. For example, the roots of chromatic polynomials (usually referred to as {\em chromatic roots}) have been of interest since the inception of chromatic polynomials, as the infamous Four Color Conjecture (now the Four Color Theorem) was equivalent to stating that $4$ is never a chromatic root of a planar graph.  While it is clear that the real chromatic roots are nonnegative (as the polynomial has coefficients that alternate in sign), it is not hard to show that $(0,1)$ is always a root-free interval for chromatic roots. Were there others? In fact, a combination of results by Thomassen \cite{thomassen} and Jackson \cite{jackson} proved that the closure of real chromatic roots is exactly the set $\{0,1\} \cup [32/27,\infty)$ (and hence, surprisingly, $(1,32/27)$ is chromatic root-free). For all-terminal reliability polynomials (the probability that a graph is connected, given that the edges are independently operational with probability $p$), the closure of their real roots \cite{1992BrownColbourn} is precisely $\{0,\} \cup [1,2]$. We remark that, in contrast to the real case, the closure of the {\em complex} chromatic roots is the entire complex plane \cite{2001Sokal}, while the closure of the complex all-terminal roots is not yet known (while it contains the unit disk centered at $z = 1$ \cite{1992BrownColbourn}, there are some roots just outside the disk \cite{2004RoyleSokal,2017BrownMolRel}).

Here we investigate the real roots of {\em domination polynomials}, the generating function for the number of dominating sets in a (finite, undirected) graph. For a graph $G$, a subset of vertices $S$ is a \emph{dominating set} of $G$ iff every vertex $v$ of $G$ is either in $S$ or adjacent to a vertex in $S$ (equivalently, $S$ is a dominating set if the \emph{closed neighbourhood} $N[v]$ of every vertex $v$ in $G$ has a non-empty intersection with $S$). The \emph{domination number} of $G$, denoted $\gamma(G)$, is the order of the smallest dominating set of $G$. For a thorough  discussion of dominating sets in graphs, see, for example, \cite{hedetniemi}.
Let $\mathcal{D}(G)$ denote the collection of dominating sets of $G$. Furthermore let $d_k(G) = |\{S \in \mathcal{D}(G): |S|=k\}|$. The \emph{domination polynomial} of $G$ is defined by

$$D(G,x) = \sum_{S \in \mathcal{D}(G)}x^{|S|} =  \sum_{k=\gamma (G)}^{|V(G)|} d_k(G)x^k.$$

The domination polynomials and their roots {\em domination roots}) have been of significant interest over the last 10 years(c.f.\ \cite{2012AlikhaniPHD}). Alikhani characterized graphs with two, three and four distinct domination roots \cite{2010DomRoots,2011AlikhaniDistinctRoots}. In \cite{2015Oboudi} Oboudi gave a degree dependent bound on the modulus of domination roots for a given graph. While Brown and Tufts \cite{2014BrownTufts} showed that domination roots are dense in the complex plane, what remains open is what the closure of the real domination roots might be. Figure \ref{fig:RealDomRoots} shows a plot of the real domination roots for all graphs of order 9. 

\begin{figure}[!h]
\centering
\includegraphics[scale=0.6]{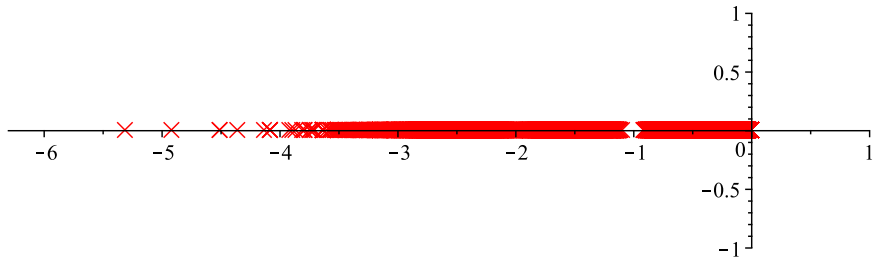}
\caption{The real domination roots for graphs of order 9}%
\label{fig:RealDomRoots}%
\end{figure}

\noindent From the plot we see that the points seem to be filling in the interval $(-4,0)$, but are more sparse to the left. Certainly, any real domination roots must be non-positive, as the polynomial has positive coefficients. Moreover, as the domination polynomial is monic, all rational roots are integers. The only known such roots are $0$ and $-2$, both for $D(K_2,x)$ (with $-1$ {\em never} a domination root \cite{2015Oboudi}), and $0$ and $-2$ are conjectured to be the {\em only} rational domination roots \cite{2010DomRoots}. With gaps at almost all rational numbers, it is natural to ask whether there are any domination root-free intervals of $(-\infty,0]$. We shall show that there are no such intervals, so the closure of the real domination roots are the entire nonpositive real axis.

\section{Closure of Real Domination Roots}

To prove our main result, we will need a graph operation, {\em graph substitution}. Let $G$ and $H$ be graphs. The graph $G[H]$, formed by substituting a copy of $H$ for every vertex of $G$, is constructed by taking a disjoint copy of $H$, $H_v$, for each vertex $v$ of $G$, and joining every vertex in $H_u$ to every vertex in $H_v$ if and only if $u$ is adjacent to $v$ in $G$. For example, the complete bipartite graph graph $K_{n,n}$ is the same as $K_{2}[\overline{K}_n]$. Domination polynomials are well-behaved with regards to graph substitution of complete graphs .

\begin{lemma}
\label{lem:GKrCompose}
\textnormal{\cite{2014BrownTufts}} Let $G$ be any graph and let $K_n$ be the complete graph on $n$ vertices. Then 
$$D(G[K_n],x)=D(G,(1+x)^n-1).$$
\end{lemma}

We now proceed to prove that real domination roots are dense in the negative real axis.

\begin{theorem}
\label{thm:RealClosure}
The closure of the real domination roots is $(-\infty,0]$.
\end{theorem}

\begin{proof}

Fix $z \in (-\infty,0]$ and $\varepsilon > 0$; we need to show that there is a domination root $z^\prime$ in the interval $(z - \varepsilon, z + \varepsilon)$. Without loss, we can assume that $z \neq -2,0$. Our proof will essentially be in two parts -- for $z \in (-2,0)$ and for $z \in (-\infty,-2)$. In either case, note that from Lemma~\ref{lem:GKrCompose}, if $z_1$ is a domination root of some graph $G$, then any solution of $(z+1)^m-1=z_1$ is a domination root (of the graph $G[K_n]$). If $m$ is an odd integer and $z_1<0$ is a domination root, then $(z_1+1)^{1/m}-1$ will be a real domination roots as well. Finally, $(z_1+1)^{1/m}-1 \in (z - \varepsilon, z + \varepsilon)$ iff $z_1 \in  ((z - \varepsilon+1)^m-1, (z + \varepsilon+1)^m-1)$, so it suffices to show that for some $m \geq 1$, the interval $((z - \varepsilon+1)^m-1, (z + \varepsilon+1)^m-1)$ contains a domination root.

\vspace{0.1in}

\noindent {\bf Case 1:} $\mathbf{z \in (-2,0)}$

\vspace{0.05in}

We shall consider two subcases that are similar in approach, splitting at $z = -1$.

\vspace{0.1in}
\noindent {\bf Subcase 1.1:} $\mathbf{z \in (-2,-1)}$
\vspace{0.05in}

We can assume that $z-\varepsilon > -2$ and $z + \varepsilon< -1$ by decreasing $\varepsilon$, so that both $(z - \varepsilon+1)^m-1$ and $(z + \varepsilon+1)^m-1$ are in $(-2,-1)$.
Observe that if we set $b = -(z - \varepsilon+1)$ and $a = -(z + \varepsilon+1)$, then $1 > b > a > 0$. Note that the interval $((z - \varepsilon+1)^m-1, (z + \varepsilon+1)^m-1) = (-b^m-1,-a^m-1)$ approach $-1$ from the left, that is, the intervals all lie to the left of $-1$, and both end points approach $-1$ as $m$ increases.  Moreover, as
\[ -b^{m+1}-1 < -a^m -1  \leftrightarrow b \left(\frac{b}{a} \right)^m > 1,\]
we conclude that if $m$ is large enough, the left end point of the next interval $(-b^{m+1}-1,a^{m+1})$ lies inside the previous interval $(-b^m-1,-a^m-1)$. It follows that the union of all the  intervals,
\[ \bigcup_{m} ((z - \varepsilon+1)^m-1, (z + \varepsilon+1)^m-1),\]
will contain an interval $(w,-1)$, with $w \in (-2,-1)$.

Now consider the domination polynomial of the complete bipartite graph $K_{k,\ell}$, which is clearly given by
\[ D(K_{k,\ell},x) = ((1+x)^k-1)((1+x)^\ell-1)+x^k+x^\ell,\]
so that 
\[ D(K_{2,\ell},x) = (1+x)^\ell (x^2+2x) + x^\ell - 2x.\]
Then 
$D(K_{2,\ell},-1) = (-1)^\ell + 2 = 1>0$. Now let $\ell$ be odd. For any $\delta \in (0,1)$,
\begin{eqnarray*}
D(K_{2,\ell},-1-\delta) & = & -\delta^\ell (\delta^2-1) - (1+\delta)^\ell + 2\delta,
\end{eqnarray*} 
which is negative for $\ell$ sufficiently large. Thus for $\ell$ large enough, there will be a real domination root in the interval $(-1-\delta,-1)$. By choosing $\delta=-w-1$ then there is a root in the interval $(w,-1) \subseteq \bigcup_{m} ((z - \varepsilon+1)^m-1, (z + \varepsilon+1)^m-1)$, so some interval $((z - \varepsilon+1)^m-1, (z + \varepsilon+1)^m-1)$ contains a real domination root.

\vspace{0.1in}
\noindent {\bf Subcase 1.2:} $\mathbf{z \in (-1,0)}$
\vspace{0.05in}

The proof of this subcase follows along that of the previous one. We can assume that $z-\varepsilon > -1$ and $z + \varepsilon< 0$, so that both $(z - \varepsilon+1)^m-1$ and $(z + \varepsilon+1)^m-1$ are in $(-1,0)$.
Observe that if we set $b = (z - \varepsilon+1)$ and $a = (z + \varepsilon+1)$, then $1 > a > b > 0$. Note that the interval $((z - \varepsilon+1)^m-1, (z + \varepsilon+1)^m-1) = (b^m-1,a^m-1)$ approach $-1$ from the right, that is, the intervals all lie to the right of $-1$, and both end points approach $-1$ monotonically as $m$ increases. Moreover, as
\[ b^{m}-1 < a^{m+1} -1  \leftrightarrow  1 < a \left(\frac{a}{b} \right)^m,\]
we conclude that if $m$ is large enough, the right end point of the next interval $(b^{m+1}-1,a^{m+1})$ lies inside the previous interval $(b^m-1,a^m-1)$. It follows that the union of all the intervals,
\[ \bigcup_{m} ((z - \varepsilon+1)^m-1, (z + \varepsilon+1)^m-1),\]
contains an interval $(-1,w)$, with $w \in (-1,0)$.


We again consider the domination polynomial of the complete bipartite graphs $K_{k,\ell}$, but with equal parts:
\[ D(K_{k,k},x) = (1+x)^{2k}-2(1+x)^k+2x^k+1.\]
Then for $k$ odd,
$D(K_{k,k},-1) = 1+2(-1)^k = -1 < 0$. For any $\delta \in (0,1)$,
\begin{eqnarray*}
D(K_{k,k},-1+\delta) & = & \delta^{2k} -2 \delta^k + 1 +  2(-1+\delta)^k,
\end{eqnarray*} 
which is positive for $k$ sufficiently large. Thus for $k$ large enough, there will be a real domination root in the interval $(-1, -1+\delta)$. By choosing $\delta=w+1$ then there is a root in the interval $(-1,w) \subseteq \bigcup_{m} ((z - \varepsilon+1)^m-1, (z + \varepsilon+1)^m-1)$, so some interval $((z - \varepsilon+1)^m-1, (z + \varepsilon+1)^m-1)$ contains a real domination root.

\vspace{0.1in}

\noindent {\bf Case 2:} $\mathbf{z \in (-\infty,-2)}$

\vspace{0.05in}

We can assume that $z+\varepsilon < -2$. Again, set $a = -(z + \varepsilon+1)$ and $b = -(z - \varepsilon+1)$; note that $b > a > 0 $. Note that the interval $((z - \varepsilon+1)^m-1, (z + \varepsilon+1)^m-1) = (-b^m-1,-a^m-1)$ has width 
\begin{eqnarray*} 
(z + \varepsilon+1)^m-1 - ((z - \varepsilon+1)^m-1) & = & b^m-a^m\\
 & = & 2(b-a) \left(b^{m-1} + b^{m-2}a + \cdots +  a^{m-1} \right)\\
 & \geq & 2\varepsilon m a^{m},
 \end{eqnarray*}
which is unbounded. Thus the width of the interval $((z - \varepsilon+1)^m-1, (z + \varepsilon+1)^m-1)$ can be arbitrarily large. We are seeking a domination root in this interval. If we can show that there is a sequence of real domination roots that tends to $-\infty$ such that the distance between successive roots is eventually bounded, then if $m$ is large enough, there will be a domination root in the interval $((z - \varepsilon+1)^m-1, (z + \varepsilon+1)^m-1)$ and we are done.

Now the domination polynomial of the star $K_{1,k}$ (yet another complete bipartite graph!) is, as noted earlier,
\[ D(K_{1,k},x) = x(x+1)^k+x^k.\]
Note that is we set $x = -R$, then 
\[ -R(1-R)^k+(-R)^k = (-1)^{k+1}(R(R-1)^k-R^k).\]
Thus setting $g_{k}(R) = R(R-1)^k-R^k$, we see that $R$ is a root of $g_k$ iff $-R$ is a root of $D(K_{1,k},x)$, so we turn our attention to $g_k$ for the time being. Note that $g_k(R) = 0$ iff 
\begin{eqnarray} 
\left( \frac{R}{R-1} \right)^k & = & R \label{mainR}
\end{eqnarray} 
Clearly on $(1,\infty)$, the left side of (\ref{mainR}), $\left( \frac{R}{R-1} \right)^k$, is a decreasing function of $R$ while the right side, $R$ is obviously increasing, there is exactly one solution to (\ref{mainR}), and hence exactly one root, say $r_k$, of $g_k$, in $(1,\infty)$ (it is the unique place where $g_k$ changes sign from negative to positive). Moreover, $r_{k+1} > r_{k}$, as 
\begin{eqnarray*} 
g_{k+1}(r_{k}) & = & r_{k}(r_{k}-1)^{k+1}-r_{k}^{k+1} \\
 & = & (r_{k}-1)^{k+1}\left( r_{k} - \left( \frac{r_{k}}{r_{k}-1}  \right)^{k+1} \right) \\
 & = &  (r_{k}-1)^{k+1}\left( r_{k} - \left( \frac{r_{k}}{r_{k}-1}  \right)^{k} \left( \frac{r_{k}}{r_{k}-1}\right)  \right)\\
 & <  & 0.
 \end{eqnarray*}
 
What about the differences between successive roots $r_{k}$? In \cite{genthetachrom} the unique root of $g_{k}$ in $(1,\infty)$ was denoted as $\mathcal{R}_{2,k}$. There it was shown that the asymptotics for $r_{k} = \mathcal{R}_{2,k}$ is given by 
\begin{eqnarray} 
r_{k} & = & \frac{k}{W(k)} + \frac{W(k)}{2(1+W(k))} + O\left( \frac{W(k)}{k} \right), \label{lambert}
\end{eqnarray}
where $W(k)$ denotes the well known {\em Lambert W function}, the inverse function to $f(w) = we^{w}$; this function is a strictly increasing of $x$, with $W(e) = 1$. 
It follows from (\ref{lambert}) that for sufficiently large $k$,
\begin{eqnarray*}
r_{k+1}-r_{k} < \frac{k+1}{W(k+1)} - \frac{k}{W(k)} + 1 + 1 < \frac{k+1}{W(k)} - \frac{k}{W(k)} + 2 < \frac{1}{W(k)} + 3 <  4
\end{eqnarray*}
Thus, returning back to the domination polynomial of stars, it follows that $-r_{1},-r_{2},\ldots,$ is a decreasing sequence of negative domination roots (of stars) that tend to $-\infty$, and that eventually have distance bounded between successive terms. It follows that {\em any} sufficiently large subinterval of the negative real axis will contain such a term, and thus we see that for large enough $m$, the interval $((z - \varepsilon+1)^m-1, (z + \varepsilon+1)^m-1)$ will contain (at least) one of these, and we have completed this case as well.

\vspace{0.1in}

In all cases, there is always a real domination root in an interval $((z - \varepsilon+1)^m-1, (z + \varepsilon+1)^m-1)$, so we conclude that the real domination roots are dense in $(-\infty,0]$.
\end{proof}


\section{Conclusion}

It may be interesting to further study the location of domination roots for various families of graphs. In particular, are the real domination roots of trees dense in $(-\infty,0]$? We have already seen in the proof of our main theorem that there are real domination roots of trees (namely stars) that are unbounded, bit we do not know if the closure is the entire nonpositive real axis..

Beyond the closure of domination roots, for each order $n$ it is natural to ask natural which graph has the smallest real domination root? It appears (see Table \ref{tab:maxmod}) that stars, which we used in case 2, have the extremal roots (and indeed the roots of largest modulus). The known approximations for the Lambert W function show that these roots of $K_{1,n}$ are roughly at $n/\ln n$.  

\begin{table}[!h]
\begin{center}
\begin{tabular}{| c | c | c |}
\hline
$n$ & Smallest real domination root (and root of maximum modulus) & Graph \\
\hline
$1$ & 0 & $K_1$\\ \hline
$2$ & 2 & $K_{1,1}$\\ \hline
$3$ & -2.618033989 & $K_{1,2}$\\ \hline
$4$ & -3.147899036 & $K_{1,3}$\\ \hline
$5$ & -3.629658127 & $K_{1,4}$\\ \hline
$6$ & -4.079595623 & $K_{1,5}$\\ \hline
$7$ & -4.506323246 & $K_{1,6}$\\ \hline
$8$ & -4.915076186 & $K_{1,7}$\\ \hline
$9$ & -5.309330065 & $K_{1,8}$\\ \hline

\end{tabular}
\end{center}
\caption{Smallest Domination Roots for $n \leq 9$}
\label{tab:maxmod}
\end{table}

\section*{Acknowledgements}
 
\noindent J. Brown acknowledges research support from Natural Sciences and Engineering Research Council of Canada (NSERC), grants RGPIN 2018-05227.


\bibliographystyle{plain}
\bibliography{DomRealRoots}

\end{document}